\documentclass[conference]{IEEEtran}
\usepackage{cite}
\usepackage{amsmath,amssymb,amsfonts}
\usepackage{graphicx}
\usepackage{textcomp}
\usepackage{xcolor}
\usepackage{graphicx, amsmath, amssymb, amsthm, tikz, algorithm, algpseudocode}
\def\BibTeX{{\rm B\kern-.05em{\sc i\kern-.025em b}\kern-.08em
    T\kern-.1667em\lower.7ex\hbox{E}\kern-.125emX}}
\usepackage[letterpaper,top=54pt,bottom=54pt,left=54pt,right=54pt]{geometry}

\newtheorem{problem}{Problem}
\newtheorem{assumption}{Assumption}
\newtheorem{definition}{Definition}
\newtheorem{lemma}{Lemma}
\newtheorem{proposition}{Proposition}
\newtheorem{theorem}{Theorem}

\begin{document}

\newgeometry{top=72pt,bottom=54pt,left=54pt,right=54pt}

\title{Lossless Convexification for Linear Systems with Piecewise Linear Controls}

\author{\IEEEauthorblockN{Shosuke Kiami}
\IEEEauthorblockA{\textit{Department of Mathematics} \\
\textit{University of California, Berkeley}\\
shokiami@berkeley.edu}
}

\maketitle

\begin{abstract}
Lossless Convexification (LCvx) is a convexification technique that transforms a class of nonconvex optimal control problems--where the nonconvexity arises from a lower bound on the control norm--into equivalent convex problems, with the goal being to apply fast polynomial-time solvers. However, to solve these infinite-dimensional problems in practice, they must first be converted into finite-dimensional problems, and it remains an open challenge to ensure the theoretical guarantees of LCvx are maintained across this discretization step. Prior work has proven guarantees for piecewise constant controls, but these methods do not extend to piecewise linear controls, which are more relevant to real world applications.

In this work, we present an algorithm that extends LCvx guarantees to piecewise linear controls. Under mild assumptions, our algorithm provably finds a solution violating the nonconvex constraints along at most $2n_x + 2$ trajectory ``edges" using $O(\log(\Delta\rho/\varepsilon))$ solver calls (where $n_x$ is the state space dimension and $\Delta\rho = \rho_{\max} - \rho_{\min}$ is the difference in our control norm bounds). A key feature is the perturbation of the control norm lower bound and the addition of rate constraints on the controls, ensuring LCvx holds along the trajectory edges. Finally, we provide numerical results demonstrating the effectiveness of our algorithm.
\end{abstract}


\section{Introduction}
Suppose we are given the following linear time-invariant dynamical system:
$$\dot{x}(t) = A_c x(t) + B_c u(t)$$
where $x(t) \in \mathbb{R}^{n_x}$ is the system state, $u(t) \in \mathbb{R}^{n_u}$ is the control input, and $A_c \in \mathbb{R}^{n_x \times n_x}$, $B_c \in \mathbb{R}^{n_x \times n_u}$ define the dynamics model. In an optimal control problem, the goal is to minimize a continuous-time cost functional subject to the system dynamics and additional path constraints. This general formulation has powerful applications to various areas such as autonomous driving \cite{xiao}, economic design \cite{weber} population dynamics \cite{barbu}, etc.

If the cost functional and set of feasible controls are convex then this yields a convex optimization problem which we hope to solve via fast polynomial-time solvers such as interior point methods \cite{nemmirovski, nocedal}. However, in many aerospace applications, a common constraint is a nonzero lower and upper bound on the control norm which introduces nonconvexity \cite{acikmese, accikmecse2011lossless, accikmecse2013lossless, harris2014lossless}. As visualized in the left-hand size of Figure \ref{fig:viz}, this creates a control feasible set resembling an annulus (or more generally, a spherical shell in higher dimensions). We use the following formulation for our problem:
\begin{problem} \label{prob:original}
    \begin{align*}
         \min_{x(t), u(t), t_f} \quad & m\left(x\left(t_f\right)\right)+\int_0^{t_f} l_c\left(\|u(t)\|\right) dt\\
        \operatorname{s.t.} \quad & \dot{x}(t)=A_c x(t)+B_c u(t)\\
        & \rho_{\min} \le \|u(t)\| \le \rho_{\max}\\
        & x(0)=x_\textnormal{init}, \quad G(x(t_f)) = 0
    \end{align*}
\end{problem}
\noindent where $m: \mathbb{R}^{n_x} \to \mathbb{R}$ is a convex $C^1$ function, $l_c: \mathbb{R} \to \mathbb{R}$ is a monotonically increasing convex $C^1$ function, $G: \mathbb{R}^{n_x} \to \mathbb{R}^{n_G}$ is an affine map, $\rho_{\min} < \rho_{\max}$, and $\|\cdot\|$ denotes the Euclidean 2-norm.

The premise of Lossless Convexification (LCvx) is to relax this nonconvex problem into an equivalent convex problem by introducing a slack variable $\sigma(t) \in \mathbb{R}$ \cite{acikmese, accikmecse2011lossless, accikmecse2013lossless, harris2014lossless}, giving the following relaxed problem:
\begin{problem} \label{prob:lcvx}
    \begin{align*}
         \min_{x(t), u(t), \sigma(t), t_f} \quad & m\left(x\left(t_f\right)\right)+\int_0^{t_f} l_c\left(\sigma(t)\right) dt\\
        \operatorname{s.t.} \quad & \dot{x}(t)=A_c x(t)+B_c u(t)\\
        & \|u(t)\| \le \sigma(t) \quad \rho_{\min} \le \sigma(t) \le \rho_{\max}\\
        & x(0)=x_\textnormal{init}, \quad G(x(t_f)) = 0
    \end{align*}
\end{problem}

Note that our control feasible set is now convex as visualized in Figure \ref{fig:viz}.

\begin{figure}[t]
    \centering
    \includegraphics[width=0.45\textwidth]{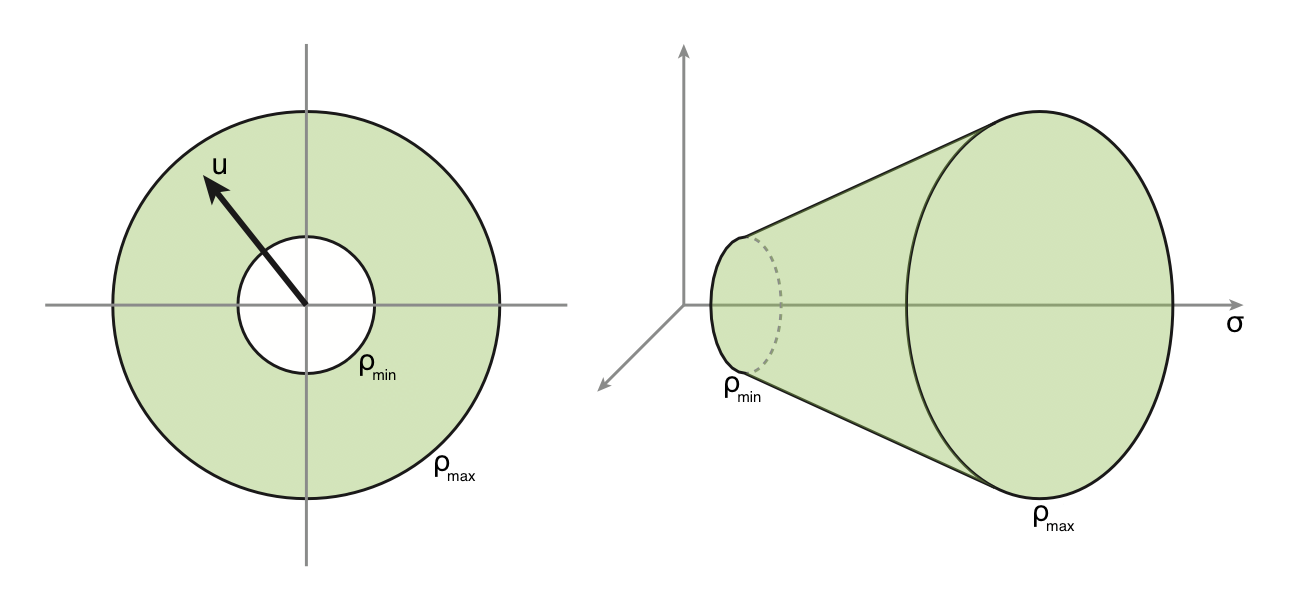}
    \caption{A geometric interpretation of lossless convexification. 
    Left: non-convex set of feasible controls. 
    Right: convexified set of feasible controls.}
    \label{fig:viz}
\end{figure}

Although the new convexified set of feasible controls could potentially allow for violation of our original nonconvex control constraints, it can be shown using Pontryagin's maximum

\restoregeometry

\noindent principle that the solution to the convexified problem always satisfies the nonconvex constraints \cite{acikmese, accikmecse2011lossless, accikmecse2013lossless, harris2014lossless, berkovitz}. In other words, this is not a relaxation in the usual sense: it does not provide an approximate solution, but an exact one.

However, in order to solve this problem in practice using solvers such as CVXPY \cite{cvxpy1, cvxpy2} or ECOS \cite{ecos}, we must first ``discretize" our trajectory into a finite-dimensional problem. The most common approach is the following:
\begin{definition} \label{def:traj}
    Given $t_f$, divide our trajectory into $N$ equal segments
    $$0 = t_1 \le t_2 \le \cdots \le t_{N+1} = t_f$$
    such that $t_{i+1} - t_i = t_f / N$ for all $i \in \{1, ..., N\}$. Let vertex $i$ denote the point $t_i$ and edge $\{i, i+1\}$ denote the interval $[t_i, t_{i+1}]$. Finally, define:
    \begin{align*}
        x &= (x_1, ..., x_{N+1}), \; x_i = x(t_i) \in \mathbb{R}^{n_x}\\
        u &= (u_1, ..., u_{N+1}), \; u_i = u(t_i) \in \mathbb{R}^{n_u}\\
        \sigma &= (\sigma_1, ..., \sigma_{N+1}), \; \sigma_i = \sigma(t_i) \in \mathbb{R}
    \end{align*}
\end{definition}

Given a discretized trajectory, there are several ways to parameterize our control input. The parameterization discussed in Luo et al. \cite{luo} is piecewise constant control, where the control along edge $\{i, i+1\}$ is held constant at $u_i$, giving us $N$ control variables (one per edge). However, this approach is often impractical, as many systems cannot instantaneously switch between control inputs. Therefore, in this paper, we focus on piecewise linear control, where the control along edge $\{i, i+1\}$ is linearly interpolated between $u_i$ and $u_{i + 1}$:
$$u(t) = \left( \frac{t_{i+1} - t}{t_{i+1} - t_i} \right) u_i + \left(\frac{t - t_{i}}{t_{i+1} - t_i} \right) u_{i + 1}$$
giving us $N+1$ control variables (one for each vertex).

Using this control parameterization, we can then discretize our dynamics as
$$x_{i + 1} = Ax_i + B_0 u_i + B_1 u_{i+1}$$
This is done using the state transition matrices
$$\Phi_A(t, t_i) = \frac{\partial x(t)}{\partial x_i}, \; \Phi_{B_0}(t, t_i) = \frac{\partial x(t)}{\partial u_i}, \; \Phi_{B_1}(t, t_i) = \frac{\partial x(t)}{\partial u_{i+1}}$$
By the chain rule, we have
\begin{align*}
    \dot{\Phi}_A(t, t_i) &=  \frac{\partial \dot{x}(t)}{\partial x_i} = A_c \Phi_A(t, t_i)\\
    \dot{\Phi}_{B_0}(t, t_i) &= \frac{\partial \dot{x}(t)}{\partial u_i} = A_c \Phi_{B_0}(t, t_i) + B_c \left( \frac{t_{i+1} - t}{t_{i+1} - t_i} \right)\\
    \dot{\Phi}_{B_1}(t, t_i) &= \frac{\partial \dot{x}(t)}{\partial u_{i+1}} = A_c \Phi_{B_1}(t, t_i) + B_c \left( \frac{t - t_{i}}{t_{i+1} - t_i} \right)
\end{align*}
with the initial conditions $\smash{\Phi_A(t_i, t_i) = I}$, $\smash{\Phi_{B_0}(t_i, t_i) = 0}$, and $\smash{\Phi_{B_1}(t_i, t_i) = 0}$. In practice, we can numerically integrate these to get
$$A = \Phi_A(t_{i+1}, t_i), \; B_0 = \Phi_{B_0}(t_{i+1}, t_i), \; B_1 = \Phi_{B_1}(t_{i+1}, t_i)$$

Finally, note that $t_f$ can be minimized via a linear search. Thus, for the remainder of this paper, we assume it is fixed as doing so distills the main contributions of our method. This yields the following discretized version of the problem:
\begin{problem} \label{prob:discrete}
    \begin{align*}
        \min_{x, u, \sigma} \quad & m\left(x_{N+1}\right) + \sum_{i = 1}^{N+1} l_i\left(\sigma_i \right) \\
        \operatorname{s.t.} \quad & x_{i+1} = A x_{i} + B_0 u_i + B_1 u_{i+1} \quad \forall i \in \{1, ..., N\} \\
        & \rho_{\min} \le \sigma_i \le \rho_{\max}, \quad \|u_i\| \le \sigma_i \quad \forall i \in \{1, ..., N+1\}\\
        & x_1 = x_\textnormal{init}, \quad G\left(x_{N+1}\right) = 0
    \end{align*}
\end{problem}

Although now our problem is tractable to numerical solvers, LCvx theory no longer offers guarantees for discretized problems. It turns out that the solution to Problem \ref{prob:discrete} may violate our original nonconvex constraints \cite{luo}. Throughout the remainder of this paper, we use the following terminology:
\begin{definition} \label{def:lcvx}
    We say that LCvx holds at vertex $i$ if
    $$\rho_{\min} \le \|u_i\| \le \rho_{\max}$$
    For piecewise linear controls, we say LCvx holds along edge $\{i, i+1\}$ if
    $$\rho_{\min} \le \|u(t)\| \le \rho_{\max} \qquad \forall t \in [t_i, t_{i+1}]$$
\end{definition}

\begin{figure}[t]
    \centering
    \includegraphics[width=0.24\textwidth]{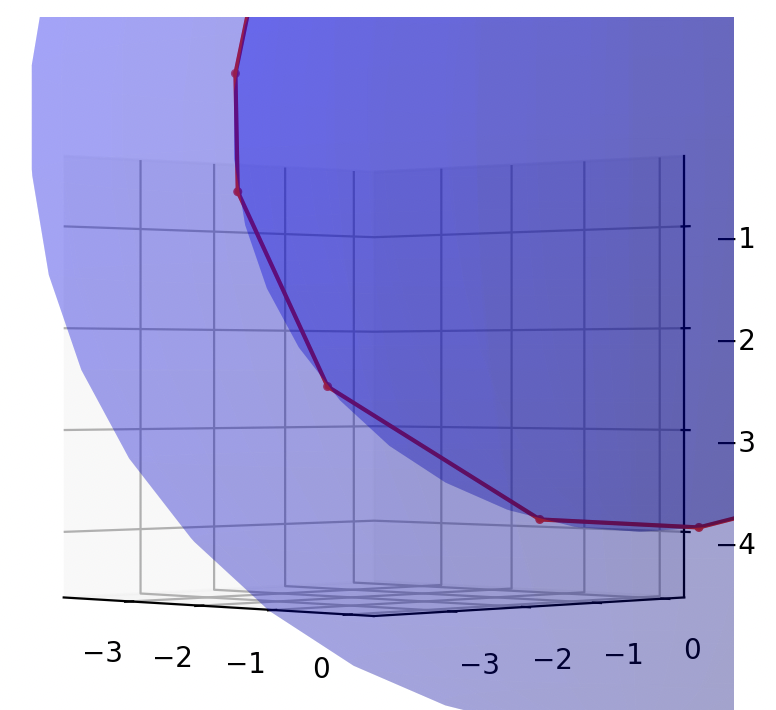}
    \includegraphics[width=0.24\textwidth]{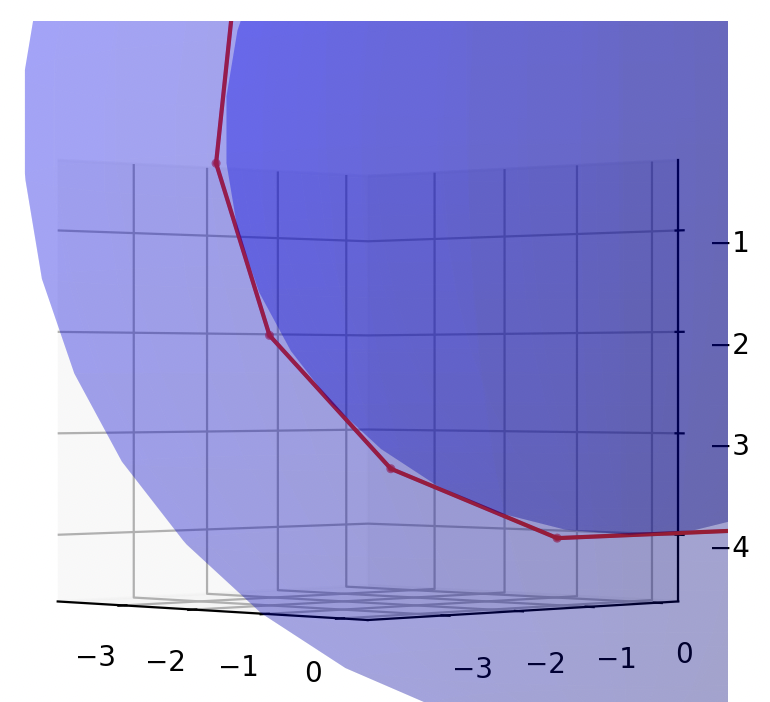}
    \caption{Illustration of why LCvx at the vertices is insufficient for piecewise linear controls. Left: LCvx holding at vertices but not edges. Right: LCvx holding along the vertices and edges.}
    \label{fig:zoom}
\end{figure}

Luo et al. show that for piecewise constant controls, solving the discretized problem ensures LCvx is violated on at most $n_x - 1$ vertices \cite{luo}. However, this result fails to extend to piecewise linear controls for the following two reasons:
\begin{enumerate}
    \item Discretizing with piecewise constant controls gives the following dynamics:
    $$x_{i+1} = Ax_i + Bu_i$$
    which is fundamentally different from our discretized dynamics.
    \item Even if we have LCvx guarantees at the vertices, this does not guarantee LCvx along the edges interpolating them (Figure \ref{fig:zoom}).
\end{enumerate}

The goal of this paper is to address these discrepancies. Section 2 examines the first point, proving an upper bound on the number of vertices violating LCvx. Section 3 discusses the second point, presenting an algorithm that extends the LCvx guarantees to the edges of the trajectory. Finally, Section 4 provides numerical results demonstrating the correctness of our algorithm on a classical example.

\section{Lossless Convexification at the Vertices}

In this section, we examine LCvx at the vertices of our discretized trajectory. Namely, our goal is to provide sufficient conditions such that LCvx is violated at a minimal number of vertices.

First, we enumerate all necessary assumptions that will be used by this section:
\begin{assumption} \label{assump:grand}
    Assume:
    \begin{enumerate}
        \item There exists a feasible trajectory such that $\|u_i\| < \rho_{\max}$ for all $i$.
        
        \item $A$ and $B_0 + AB_1$ form a controllable pair, i.e.
        $$\textnormal{rank}\begin{bmatrix} (B_0 + AB_1) & \cdots & A^{n_x-1}(B_0 + AB_1) \end{bmatrix} = n_x$$
        
        \item Let $z = (x, u, \sigma) \in \mathbb{R}^{n_z}$ and express our equality constraints as $H(z) = 0$ where affine $H: \mathbb{R}^{n_z} \to \mathbb{R}^{n_H}$. Then
        $$\textnormal{rank}\nabla H(z) = n_H$$
        
        \item For any solution $(x^*, u^*, \sigma^*)$, if
        $$x_{N+1}^* = \arg\min_{x} m(x) \textnormal{ s.t. } G(x) = 0$$
        then there exists some $\|u_i^*\| > 0$.
    \end{enumerate}
\end{assumption}

All of these assumptions are rather mild. Assumption \ref{assump:grand}.1 says we should be able to get from $x_1 = x_\textnormal{init}$ to $G(x_{N+1}) = 0$ without exerting maximal control effort. Assumption \ref{assump:grand}.2 intuitively requires our dynamical system to be controllable. Assumption \ref{assump:grand}.3 is a standard assumption in optimal control theory ensuring independence of our constraints \cite{luo}. Assumption \ref{assump:grand}.4 requires our system to exert some non-zero control effort to satisfy our terminal condition optimally.

Now, we define a couple of useful sets representing the inequality constraints in our problem:
\begin{definition} \label{def:v}
    Define our control feasible set to be:
    $$V = \{(u, \sigma) \in \mathbb{R}^{n_u} \times \mathbb{R} : \rho_{\min} \le \sigma \le \rho_{\max}, \|u\| \le \sigma)$$
    Denote taking a slice of $V$ at a fixed $\sigma \in [\rho_{\min}, \rho_{\max}]$ as:
    $$V(\sigma) = \{u \in \mathbb{R}^{n_u} : \|u\| \le \sigma)$$
\end{definition}

We now begin building our conditions with the following fact.
\begin{lemma} \label{lem:boundary}
    If $u_i$ is on the boundary of $V(\sigma_i)$, then LCvx holds at vertex $i$.
\end{lemma}

\begin{proof}
    By Definition \ref{def:v}, if $u_i$ is on the boundary of $V(\sigma_i)$ then $\|u_i\| = \sigma_i$. Since $\rho_{\min} \le \sigma_i \le \rho_{\max}$, we have $\rho_{\min} \le \|u_i\| \le \rho_{\max}$. Our result follows from Definition \ref{def:lcvx}.
\end{proof}

Intuitively, this result states that our control values should lie on the outer boundary of the convexified set in Figure \ref{fig:viz} in order to guarantee LCvx.

To analyze when this is the case, we begin by writing down the Lagrangian of Problem \ref{prob:discrete}:
\begin{definition} \label{def:lagrangian}
    \begin{align*}
        &\mathcal{L}(x, u, \sigma, \eta, \mu_1, \mu_2) =\\
        &\qquad \qquad \qquad m\left(x_{N+1}\right) + \sum_{i=1}^{N+1} l_i\left( \sigma_i \right)\\
        &\qquad \qquad \qquad + \sum_{i=1}^N \eta_i^\top \left(-x_{i+1} + A x_i + B_0 u_{i} + B_1 u_{i+1} \right)\\
        &\qquad \qquad \qquad + \mu_1^\top \left( x_1 - x_\textnormal{init} \right) + \mu_2^\top G(x_{N+1})\\
        &\qquad \qquad \qquad + \sum_{i=1}^{N+1} I_V(u_i, \sigma_i)
    \end{align*}
    where
    $$I_V(u_i, \sigma_i) = \begin{cases} 
        0 & \textnormal{if } (u_i, \sigma_i) \in V\\
        \infty & \textnormal{o.w.}
    \end{cases}$$
    $\mu_1 \in \mathbb{R}^{n_x}$, $\mu_2 \in \mathbb{R}^{n_G}$, and $\eta = (\eta_1, ..., \eta_N)$ with $\eta_i \in \mathbb{R}^{n_x}$.
\end{definition}

We make the following observation:
\begin{lemma} \label{lem:duality}
    For any primal solution $(x^*, u^*, \sigma^*)$, there exists dual variables $\eta^*, \mu_1^*, \mu_2^*$ such that
    $$(x^*, u^*, \sigma^*) = \arg\min_{x, u, \sigma} \mathcal{L}(x, u, \sigma, \eta^*, \mu_1^*, \mu_2^*)$$
\end{lemma}

\begin{proof}
    By Assumption \ref{assump:grand}.1, we have some a feasible trajectory such that $\|u_i\| < \rho_{\max}$. For all $i$, let
    $$\sigma_i = \frac{\max\{\|u_i\|, \rho_{\min} \} + \rho_{\max}}{2}$$
    
    Since $\rho_{\min} < \rho_{\max}$, this gives $\rho_{\min} < \sigma_i < \rho_{\max}$ and $\|u_i\| < \sigma_i$. Thus, we have a feasible trajectory satisfying Slater's condition. By section 5.9.1 and 5.9.2 of \cite{boyd}, we have our result.
\end{proof}

We can use Lemma \ref{lem:duality} along with KKT conditions of optimality to generate the following useful relationships between our dual variables.
\begin{lemma} \label{lem:etas}
    We have:
    \begin{enumerate}
        \item $\eta_{i-1}^* = A^\top \eta_i^* \quad \forall i \in \{2, ..., N\}$
        \item $\eta_N^* = \nabla m(x_{N+1}^*) + \mu_2^\top \nabla G(x_{N+1}^*)$
    \end{enumerate}
\end{lemma}

\begin{proof}
    From Lemma \ref{lem:duality}, we have
    $$x^* = \arg\min_{x} \mathcal{L}(x, u^*, \sigma^*, \eta^*, \mu_1^*, \mu_2^*)$$
    
    By the KKT conditions of optimality, we know that $\smash{\tfrac{\partial}{\partial x_i} \mathcal{L}(x, u^*, \sigma^*, \eta^*, \mu_1^*, \mu_2^*) = 0}$ for all $i$. Thus, isolating each component, we get\\
    $i = 1$:
    $$\mu_1^\top + \eta_1^{* \top} A = 0$$
    $2 \le i \le N$:
    $$\eta_i^{* \top} A - \eta_{i-1}^{* \top} = 0$$
    $i = N+1$:
    $$\nabla m(x_{N+1}^*) + \mu_2^\top \nabla G(x_{N+1}^*) - \eta_N^* = 0$$
\end{proof}

Using Lemma \ref{lem:duality} and \ref{lem:etas}, we provide our first set of sufficient conditions for LCvx to hold at a given vertex.
\begin{proposition} \label{prop:conditions}
    LCvx holds at vertex $i$ if:
    $$\begin{cases} 
        B_0^\top \eta_1^* \ne 0 & \textnormal{if } i = 1\\
        B_1^\top \eta_N^* \ne 0 & \textnormal{if } i = N + 1\\
        (B_0 + AB_1)^\top \eta_i^* \ne 0 & \textnormal{o.w.}
    \end{cases}$$
\end{proposition}

\begin{proof}
    See Appendix A.
    \renewcommand{\qedsymbol}{}
\end{proof}

To further develop these conditions into a more usable form, we can leverage the structure imposed by Assumptions \ref{assump:grand}.2 and \ref{assump:grand}.3. However, to do so, we must also first introduce a small perturbation to our dynamics model. More specifically, define a perturbation to $A$ as follows:
\begin{definition} \label{def:perturbation}
    Let $A = P J P^{-1}$ be the Jordan normal form of $A$ and $\lambda_1, ..., \lambda_{d}$ be the distinct eigenvalues of $A$. Given $q \in \mathbb{R}^d$, define a perturbation of $A$ as
    $$\tilde{A}(q) = P \tilde{J} P^{-1}$$
    where $\tilde{J}$ is obtained by replacing the eigenvalues of $J$ with
    $$(\tilde{\lambda}_1, ..., \tilde{\lambda}_d) = (\lambda_1, ..., \lambda_d) + (q_1, ..., q_d)$$
    (leaving Jordan blocks unaffected).
\end{definition}

Replacing $A$ in Problem \ref{prob:discrete} with $\tilde{A}(q)$ gives us our new problem:
\begin{problem} \label{prob:perturbed}
    \begin{align*}
        \min _{x, u, \sigma} \quad & m\left(x_{N+1}\right) + \sum_{i = 1}^{N+1} l_i\left(\sigma_i \right)\\
       \operatorname{s.t.} \quad & x_{i+1} = \tilde{A}(q) x_{i} + B_0 u_i + B_1 u_{i+1} \quad \forall i \in \{1, ..., N\}\\
        & \rho_{\min} \le \sigma_i \le \rho_{\max}, \quad \|u_i\| \le \sigma_i \quad \forall i \in \{1, ..., N+1\}\\
        & x_1 = x_\textnormal{init}, \quad G\left(x_{N+1}\right) = 0
    \end{align*}
\end{problem}

Before we discuss how Assumptions \ref{assump:grand}.2 and \ref{assump:grand}.3 and this perturbation contribute to the LCvx guarantees, we note that the perturbation does not affect our previous assumptions and has a negligible affect on the optimal trajectory:
\begin{proposition} \label{prop:perturbed_traj}
    For any $\delta > 0$, there exists some $\varepsilon_A > 0$ such that if $\forall q_i \in [-\varepsilon_A, \varepsilon_A]$, then Problem \ref{prob:perturbed} still satisfies Assumption \ref{assump:grand} and
    \begin{enumerate}
        \item $\|G (\tilde{x}_{N+1} ) \| \le \delta$
        \item $| m (\tilde{x}_{N+1} ) + \sum_{i=1}^{N+1} l_i ( \|u_i^* \| ) - m^* | \le \delta$
    \end{enumerate}
    where $\tilde{x}_{N+1}$ is achieved by taking $x_1 = x_\textnormal{init}$ and iterating $x_{i+1} = A x_{i} + B_0 u_i^* + B_1 u_{i+1}^*$, and $m^*$ is the optimal cost of solving Problem \ref{prob:discrete} (without perturbation).
\end{proposition}

\begin{proof}
    Lemma 16 of \cite{luo} tells us that there exists some sufficiently small $\varepsilon_1 > 0$ such that Assumption \ref{assump:grand} still holds. Theorem 19 of \cite{luo} tells us that there exists some sufficiently small $\varepsilon_2 > 0$ that gives us the remainder of the result. Taking
    $$\varepsilon_A = \min \left\{\varepsilon_1, \varepsilon_2 \right\}$$
    gives our desired result.
\end{proof}

With this, we give our main result regarding LCvx guarantees at the vertices.
\begin{theorem} \label{thm:big}
    Suppose we solve Problem \ref{prob:perturbed} by sampling $q \in \mathbb{R}^d$ with
    $$q_i \overset{\mathrm{i.i.d.}}{\sim} \textnormal{Unif} [-\varepsilon_A, \varepsilon_A]$$
    If $\eta_N^* \ne 0$, then LCvx is violated on at most $n_x + 1$ vertices with probability 1.
\end{theorem}

\begin{proof}
    Given Assumptions \ref{assump:grand}.2 and \ref{assump:grand}.3, by Theorem 18 of \cite{luo}, we have that if $\eta_N^* \ne 0$, then $(B_0 + AB_1)^\top \eta_i^* = 0$ on at most $n_x - 1$ vertices with probability 1. By Proposition \ref{prop:conditions}, this guarantees LCvx holds at vertex $i$ when $2 \le i \le N$, but does not guarantee anything for $i = 1$ or $i = N + 1$. Thus, LCvx is violated on at most $n_x + 1$ vertices with probability 1.
\end{proof}

At this point it is worth noting that, in practice, the need for such a perturbation is almost never necessary, and the discussion of this perturbation method is purely for theoretical guarantees \cite{luo}.

For the remainder of this section, we discuss when $\eta_N^* \ne 0$. To do so, we first introduce the following lemma:
\begin{lemma} \label{lem:eta_N}
    If $\eta_N^* = 0$, then
    \begin{enumerate}
        \item $\sigma_i^* = \rho_{\min} \quad \forall i \in \{1, ..., N + 1\}$
        \item $x_{N+1}^* = \arg\min_{x} m(x) \textnormal{ s.t. } G(x) = 0$
    \end{enumerate}
\end{lemma}

\begin{proof}
    Repeatedly applying Lemma \ref{lem:etas}.1, we have $\eta_i^* = A^{\top (N - i)} \eta_N^*$. Thus, if $\eta_N^* = 0$, then $\eta_i^* = 0$ for all $i \in \{1, ..., N\}$. From Lemma \ref{lem:duality}, we have
    \begin{align*}
        (u^*, \sigma^*) &= \arg\min_{\sigma, u} \mathcal{L}(x^*, u, \sigma, \eta^*, \mu_1^*, \mu_2^*)\\
        &= \arg\min_{\sigma, u} \sum_{i=1}^{N+1} l_i\left( \sigma_i \right) + \sum_{i=1}^{N+1} I_V(u_i, \sigma_i)
    \end{align*}
    Since $l_i$ is monotonically increasing, we have $\sigma_i^* = \rho_{\min}$ and $u_i^* \in V(\rho_{\min})$ for all $i \in \{1, ..., N + 1\}$.
    
    Next, by Lemma \ref{lem:etas}.2, we have
    $$\eta_N^* = \nabla m(x_{N+1}^*) + \mu_2^\top \nabla G(x_{N+1}^*) = 0$$
    
    We see that $x_{N+1}^*$ satisfies the KKT conditions of optimality for the convex optimization problem
    $$\min_{x} m(x) \textnormal{ s.t. } G(x) = 0$$
    and is thus a minimizer.
\end{proof}

Using this lemma, we can now give a more intuitive condition for when $\eta_N \ne 0$. For a fixed Problem \ref{prob:perturbed}, consider different values of $\rho_{\min}$:
\begin{proposition} \label{prop:rho_max}
    There exists $\rho \in (0, \rho_{\max}]$ such that if $\rho_{\min} < \rho$ then $\eta_N^* \ne 0$.
\end{proposition}

\begin{proof}
    For this proof, we divide into two cases:
    
    Case 1: $x_{N+1}^* \ne \arg\min_{x} m(x) \textnormal{ s.t. } G(x) = 0$, implying $\eta_N^* \ne 0$ by the contrapositive of Lemma \ref{lem:eta_N}. Since this holds for any $\rho_{\min} < \rho_{\max}$, we can simply take $\rho = \rho_{\max}$.
    
    Case 2: $x_{N+1}^* = \arg\min_{x} m(x) \textnormal{ s.t. } G(x) = 0$. By Assumption \ref{assump:grand}.4, there is some $\|u_i^*\| > 0$. Choose $\rho = \|u_i^*\| / 2$. Then for all $\rho_{\min} < \rho$,
    $$\rho_{\min} < \rho < \|u_i^*\| \le \sigma_i^*$$
    Thus, by the contrapositive of Lemma \ref{lem:eta_N}, $\eta_N^* \ne 0$.
\end{proof}
Intuitively, Proposition \ref{prop:rho_max} says that if $\rho_{\min}$ is sufficiently small, then $\eta_N^* \ne 0$. As a note, \cite{luo} makes a similar remark, observing that $\eta_N^* \ne 0$ if $t_f$ is sufficiently small. Using Proposition \ref{prop:rho_max}, we include the following helper definition:
\begin{definition} \label{def:rho_max}
    Let
    $$\rho_{\min}^+ = \sup\{\rho \in (0, \rho_{\max}] : \rho_{\min} < \rho \implies \eta_N^* \ne 0\}$$
\end{definition}
This upper bound will prove useful in our algorithm development in the next section.

\section{Lossless Convexification along the Edges}
At this point, we have proven sufficient conditions for LCvx to hold at the vertices of our trajectory. In this section, we present an algorithm that extends this guarantee to the edges of our trajectory.

First, we again define a new problem. Given a perturbed lower bound on our control norm $\tilde{\rho}_{\min} \in [\rho_{\min}, \rho_{\max})$, define:
\begin{problem} \label{prob:delta}
    \begin{align*}
        \min _{x, u, \sigma} \quad & m\left(x_{N+1}\right) + \sum_{i = 1}^{N+1} l_i\left(\sigma_i \right)\\
       \operatorname{s.t.} \quad & x_{i+1} = \tilde{A}(q) x_{i} + B_0 u_i + B_1 u_{i+1}, \quad \forall i \in \{1, ..., N\} \\
        & \tilde{\rho}_{\min} \le \sigma_i \le \rho_{\max}, \quad \|u_i\| \le \sigma_i \quad \forall i \in \{1, ..., N+1\}\\
        & x_1 = x_\textnormal{init}, \quad G\left(x_{N+1}\right) = 0\\
        & \|u_{i+1}-u_i\| \le \delta(\tilde{\rho}_{\min}) \quad \forall i \in \{1, ..., N\}
    \end{align*}
\end{problem}
where
$$\delta(\tilde{\rho}_{\min}) = 2\sqrt{\tilde{\rho}_{\min}^2 - \rho_{\min}^2}$$

Note that $\tilde{\rho}_{\min}$ replaced $\rho_{\min}$ and we introduced an additional rate constraint $\|u_{i+1}-u_i\| \le \delta(\tilde{\rho}_{\min})$. The purpose of this modification is illustrated in the following result:
\begin{lemma} \label{lem:edges}
    If $\tilde{\rho}_{\min} \le \|u_i\| \le \rho_{\max}$ and $\tilde{\rho}_{\min} \le \|u_{i+1}\| \le \rho_{\max}$ with $\|u_{i+1}-u_i\| \le \delta(\tilde{\rho}_{\min})$, then $\rho_{\min} \le \|u(t)\| \le \rho_{\max}$ along edge $\{i, i+1\}$.
\end{lemma}

\begin{proof}
    See Appendix B.
    \renewcommand{\qedsymbol}{}
\end{proof}

Although this new rate constraint seems helpful, its addition fundamentally changes the structure of our optimization problem and thus we do not have any of the LCvx guarantees at the vertices from the previous section. In order to address this, we begin with the following assumption:

\begin{assumption} \label{assump:existence}
    There exists $\rho \in [0, \rho_{\min}^+)$ such that if we solve Problem \ref{prob:delta} without the rate constraint and $\tilde{\rho}_{\min} > \rho$, then
    $$\delta(\tilde{\rho}_{\min}) > \max_{1 \le i \le N} \|u_{i+1}^* - u_i^*\|$$
\end{assumption}

In other words, if $\tilde{\rho}_{\min}$ is sufficiently large then the rate constraint is redundant (i.e. introduces no additional restrictions to our problem), ensuring all previous results still hold. We purposefully weaken the implication in the above assumption to define an even smaller lower bound:
\begin{definition} \label{def:rho_min}
    Let
    $$\rho_{\min}^- = \inf\{\rho \in [0, \rho_{\max}) : \tilde{\rho}_{\min} > \rho \implies \text{Theorem \ref{thm:big} holds}\}$$
\end{definition}

We summarize the previous results into the following proposition establishing the existence of an interval over which we have LCvx guarantees along the edges:
\begin{proposition} \label{prop:interval}
    There exists
    $$\rho_{\min} \le \rho_{\min}^- < \rho_{\min}^+ \le \rho_{\max}$$
    such that if we solve Problem \ref{prob:delta} with $\tilde{\rho}_{\min} \in (\rho_{\min}^-, \rho_{\min}^+)$ (and $q \in \mathbb{R}^d$ sampled according to Theorem \ref{thm:big}) then LCvx is volated along at most $2n_x + 2$ edges.
\end{proposition}

\begin{proof}
    Proposition \ref{prop:rho_max} and Assumption \ref{assump:existence} proves existence of $\rho_{\min} \le \rho_{\min}^- < \rho_{\min}^+ \le \rho_{\max}$. By Definition \ref{def:rho_max} and \ref{def:rho_min}, if $\tilde{\rho}_{\min} \in (\rho_{\min}^-, \rho_{\min}^+)$ then we can invoke Theorem \ref{thm:big} with $\eta_N^* \ne 0$. Thus, solving Problem \ref{prob:delta} violates $\tilde{\rho}_{\min} \le \|u_i^*\| \le \rho_{\max}$ on at most $n_x + 1$ vertices.
    
    Lemma \ref{lem:edges} only holds when both $\tilde{\rho}_{\min} \le \|u_i^*\| \le \rho_{\max}$ and $\tilde{\rho}_{\min} \le \|u_{i+1}^*\| \le \rho_{\max}$. Thus, if our nonconvex constraint is violated at vertex $i$, then LCvx may be violated along edges $\{i-1, i\}$ and/or $\{i, i+1\}$. Since there are at most $n_x + 1$ vertices where this can happen, we have at most $2n_x + 2$ neighboring edges where LCvx may be violated.
\end{proof}

See Figure \ref{fig:rho} for a visualization of $(\rho_{\min}^-, \rho_{\min}^+)$ on a real example. Finally, given the existence of this ``feasible" interval, we can search for such an interval using the previous conditions and then minimize the cost within the interval using a ternary search. This gives rise to Algorithm \ref{alg:pllcvx}.

\begin{algorithm}[t]
    \caption{Picewise Linear Lossless Convexification} \label{alg:pllcvx}
    \begin{algorithmic}[1]
        \Require $A_c, B_c, \rho_{\min}, \rho_{\max}, \varepsilon_A, \varepsilon$
        \State $(A, B_0, B_1) \gets \text{integrate } A_c \text{ and } B_c$
        \State $q \gets \text{sample } q_i \sim \textnormal{Unif} [-\varepsilon_A, \varepsilon_A] \text{ for } i \in \{1, ..., d\}$
        \State $\tilde{A}(q) \gets \text{perturb } A \text{ with } q$
        \State $\rho_\text{low} \gets \rho_{\min}$
        \State $\rho_\text{high} \gets \rho_{\max}$
        \While{$\rho_\text{high} - \rho_\text{low} > \varepsilon$}
            \State $\rho_1 \gets \rho_\text{low} + (\rho_\text{high} - \rho_\text{low})/3$
            \State $\rho_2 \gets \rho_\text{high} - (\rho_\text{high} - \rho_\text{low})/3$
            \State $(\text{cost}_1^*, u_1^*, \eta_1^*) \gets \text{solve Problem \ref{prob:delta} for } \tilde{\rho}_{\min} = \rho_1$
            \State $(\text{cost}_2^*, u_2^*, \eta_2^*) \gets \text{solve Problem \ref{prob:delta} for } \tilde{\rho}_{\min} = \rho_2$
            \If{$(\eta_1^*)_N \ne 0 \land \sum_i \mathbf{1}(\|(u_1^*)_i\| < \rho_{\min}) > n_x + 1$}
                \State $\rho_\text{low} \gets \rho_1$
            \ElsIf{$(\eta_2^*)_N = 0$}
                \State $\rho_\text{high} \gets \rho_2$
            \Else
                \If{$\text{cost}_1^* > \text{cost}_2^*$}
                    \State $\rho_\text{low} \gets \rho_1$
                \Else
                    \State $\rho_\text{high} \gets \rho_2$
                \EndIf
            \EndIf
        \EndWhile
        \State $(x^*, u^*) \gets \text{solve Problem \ref{prob:delta} for } \tilde{\rho}_{\min} = (\rho_\text{low} + \rho_\text{high}) / 2$
        \State \Return $(x^*, u^*)$
    \end{algorithmic}
\end{algorithm}

\begin{theorem} \label{thm:algo}
    Given a sufficiently small $\varepsilon > 0$, Algorithm \ref{alg:pllcvx} finds a minimizing solution to Problem \ref{prob:delta} violating LCvx along at most $2n_x + 2$ edges in the trajectory in
    $$O\left(\log\left(\frac{\rho_{\max} - \rho_{\min}}{\varepsilon}\right)\right)$$
    calls to our solver.
\end{theorem}

\begin{proof}
    Let $\tilde{\rho}_{\min} \in [\rho_{\min}, \rho_{\max}]$. If $\eta_N^* \ne 0$ and LCvx is violated at more than $n_x + 1$ then Theorem 1 is violated and thus $\tilde{\rho}_{\min} \le \rho_{\min}^-$ by Definition \ref{def:rho_min}. If $\eta_N = 0$ then $\tilde{\rho}_{\min} \ge \rho_{\min}^+$ by definition \ref{def:rho_max}. Furthermore, by Proposition \ref{prop:interval}, we can initialize our ternary search with $\rho_\text{low} = \rho_{\max}$ and $\rho_\text{high} = \rho_{\max}$.

    Furthermore, by section 5.6.1 of \cite{boyd}, our cost is convex (i.e.  unimodal) with respect to $\tilde{\rho}_{\min}$. Thus, once $\rho_\text{low}$ and $\rho_\text{high}$ fall within our ``feasible" interval (guaranteed by sufficiently small $\varepsilon > 0$), our algorithm is guaranteed to converge to the optimal cost in $O(\log(\Delta\rho/\varepsilon))$ steps.
\end{proof}

\section{Results}

In this section, we apply our algorithm to the classical double-integrator trajectory optimization problem, where the system is an agent controlled solely through acceleration and the objective is to minimize the total control effort. Let $x = (p_x, p_y, p_z, v_x, v_y, v_z)^\top$ be our state vector and $u = (a_x, a_y, a_z)^\top$ be our control vector where $p \in \mathbb{R}^3$ is position, $v \in \mathbb{R}^3$ is velocity, and $a \in \mathbb{R}^3$ is acceleration. Our dynamical system can be expressed as:

$$\begin{bmatrix} \dot{p_x} \\ \dot{p_y} \\ \dot{p_z} \\ \dot{v_x} \\ \dot{v_y} \\ \dot{v_z} \end{bmatrix} = \underbrace{\begin{bmatrix} 0 & 0 & 0 & 1 & 0 & 0 \\ 0 & 0 & 0 & 0 & 1 & 0 \\ 0 & 0 & 0 & 0 & 0 & 1 \\ 0 & 0 & 0 & 0 & 0 & 0 \\ 0 & 0 & 0 & 0 & 0 & 0 \\ 0 & 0 & 0 & 0 & 0 & 0 \end{bmatrix}}_{A_c} \begin{bmatrix} p_x \\ p_y \\ p_z \\ v_x \\ v_y \\ v_z \end{bmatrix} + \underbrace{\begin{bmatrix} 0 & 0 & 0 \\ 0 & 0 & 0 \\ 0 & 0 & 0 \\ 1 & 0 & 0 \\ 0 & 1 & 0 \\ 0 & 0 & 1 \end{bmatrix}}_{B_c} \begin{bmatrix} a_x \\ a_y \\ a_z \end{bmatrix}$$

Consider the following trajectory optimization problem:
\begin{align*}
    \min_{x(t), u(t)} \quad & 100\|x(4) - (10, 10, 10, 0, 0, 0)^\top \| + \int_0^{4} \|u(t)\|^2 dt\\
    \operatorname{s.t.} \quad & \dot{x}(t)=A_c x(t) + B_c u(t)\\
    & 4 \le \|u(t)\| \le 6 \quad \forall t \in [0, 4]\\
    & x(0) = (0, 0, 0, 0, 0, 10)^\top
\end{align*}

First, it can be verified that all of the conditions for Assumptions \ref{assump:grand} and \ref{assump:existence} are met. In particular with regards to Assumption \ref{assump:existence}, if we solve Problem \ref{prob:delta} without the rate constraint and $\tilde{\rho}_{\max} = 4.5$ then
$$\delta(\tilde{\rho}_{\max}) \approx 4.123 > 3.322 \approx \max_{1 \le i \le N} \|u_{i+1}^* - u_i^*\|$$
and this inequality continues to hold for all $\tilde{\rho}_{\max} > 4.5$.

Solving the problem using piecewise constant controls and the single-shot method from Luo et al. with $N = 16$ yields the trajectory shown in Figure \ref{fig:zoh}. Meanwhile, solving the problem using piecewise linear controls and Algorithm 1 with $N = 16$ and $\varepsilon = 10^{-3}$ yields the trajectory shown in Figure \ref{fig:foh}. In the control plot for our method, observe how the linear interpolation of the control values appear tangent to the inner ball, demonstrating how the controls are satisfying the norm lower bound while maintaining minimal cost. A zoomed in version is shown on the right-hand side of Figure \ref{fig:zoom}.

Also, as seen in Figure \ref{fig:rho}, $\rho_{\min}^- \approx 4.026$ and $\rho_{\min}^+ \approx 5.105$. Note that the plotted points in Figure \ref{fig:rho} are colored red if LCvx is violated at more than $n_x + 1$ vertices, orange if $\eta_N^* = 0$, and blue otherwise. Our algorithm converges to $\tilde{\rho}_{\min} \approx 4.098$. When $\tilde{\rho}_{\min} = 4.025$ and $\tilde{\rho}_{\max} = 5.5$, LCvx is being violated as expected.

\begin{figure*}[t]
    \centering
    \includegraphics[width=0.3\textwidth]{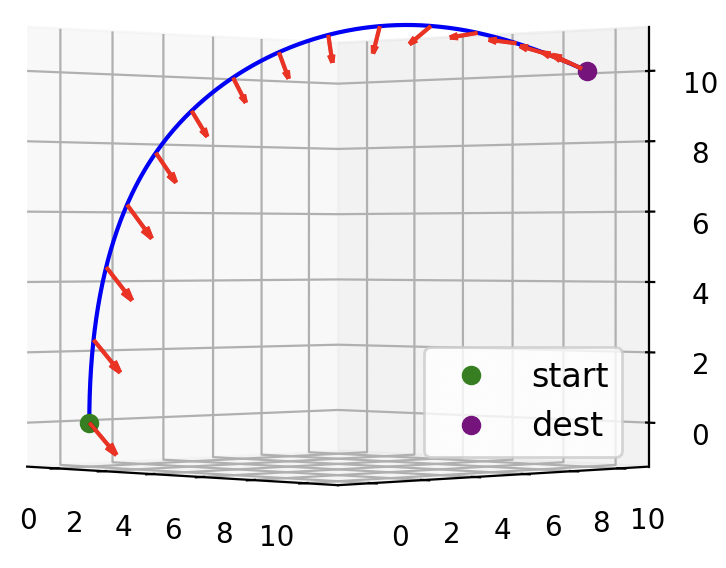}
    \includegraphics[width=0.3\textwidth]{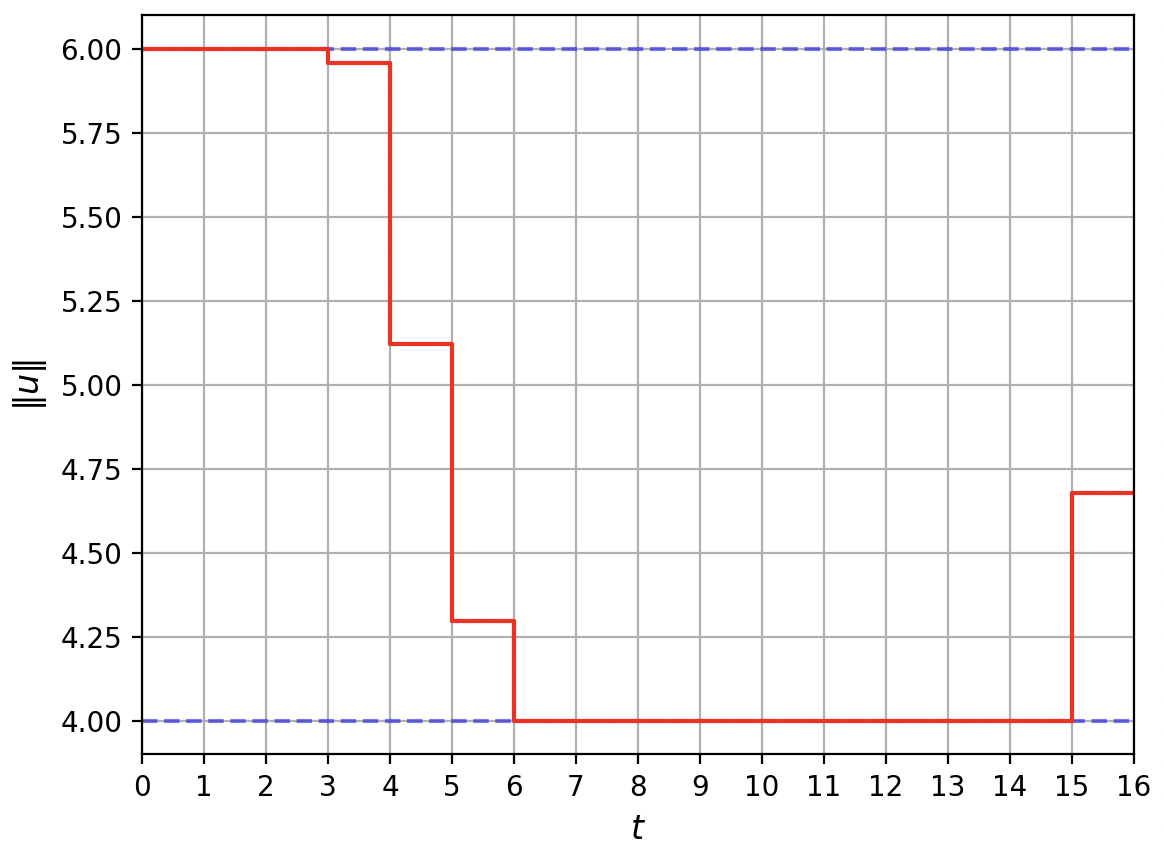}
    \includegraphics[width=0.3\textwidth]{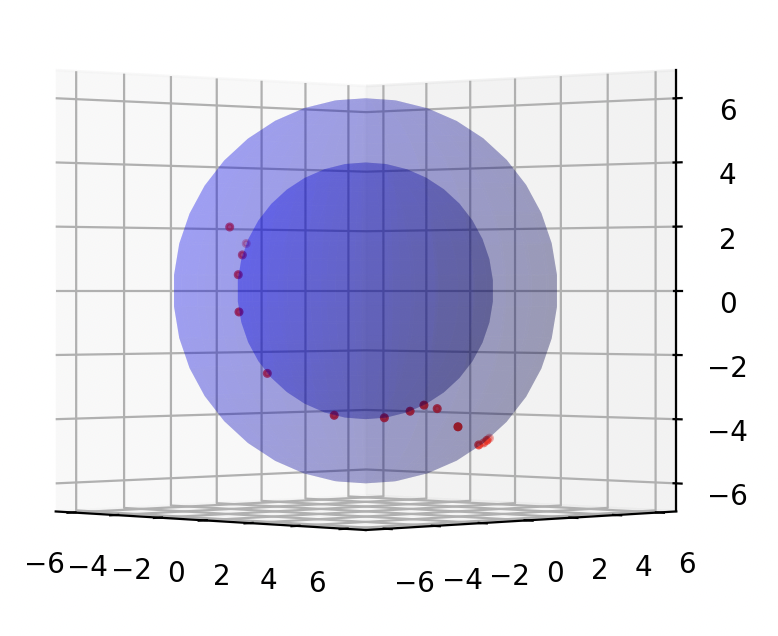}
    \caption{Results for LCvx with piecewise constant controls. Left: optimal trajectory. Middle: control magnitudes. Right: control plot.}
    \label{fig:zoh}
\end{figure*}

\begin{figure*}[t]
    \centering
    \includegraphics[width=0.3\textwidth]{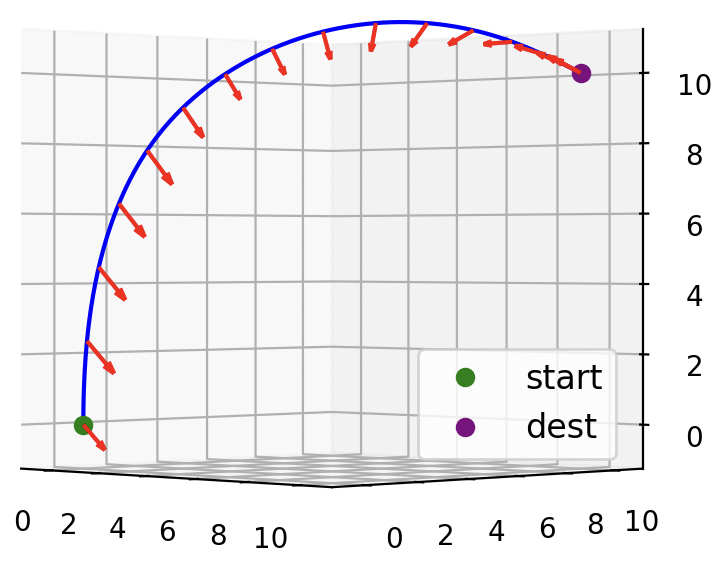}
    \includegraphics[width=0.3\textwidth]{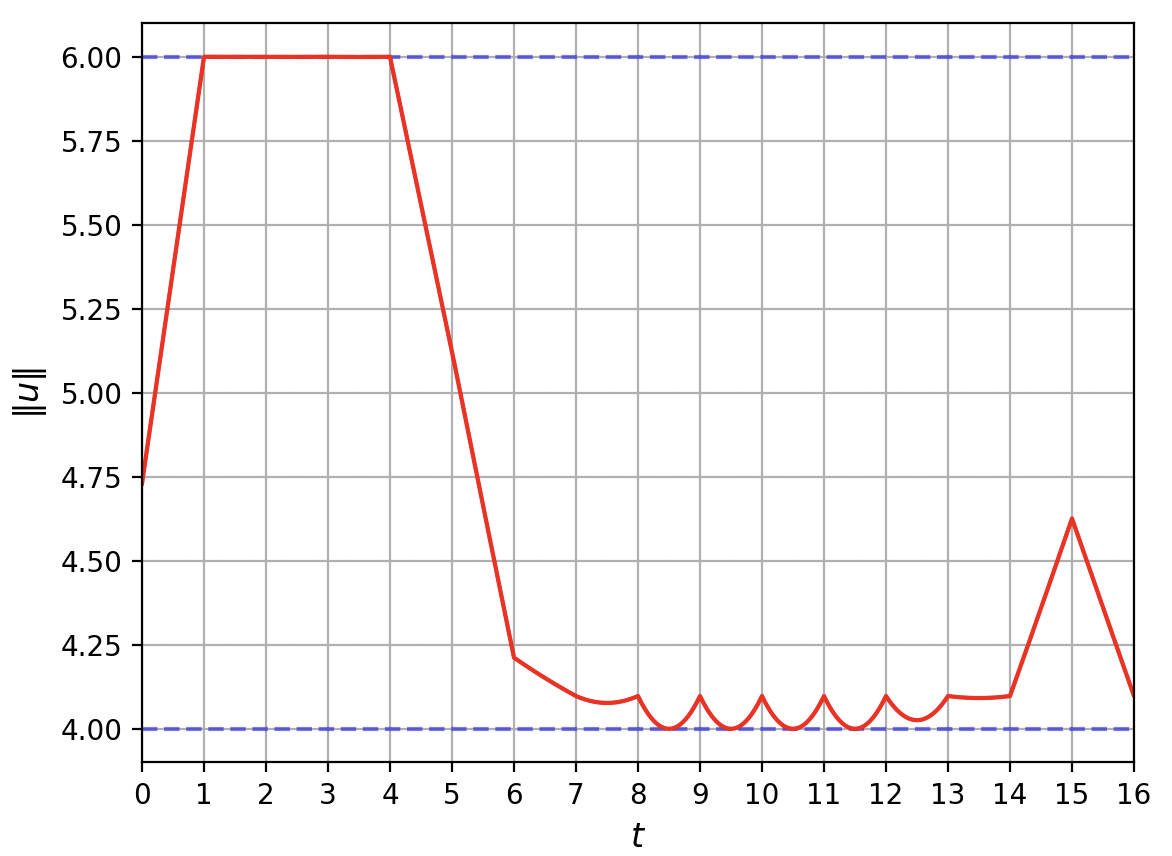}
    \includegraphics[width=0.3\textwidth]{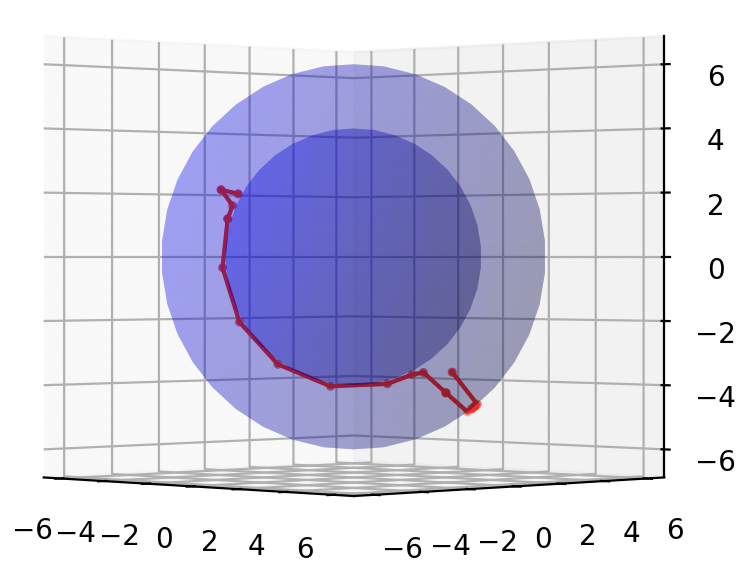}
    \caption{Results for LCvx with piecewise linear controls. Left: optimal trajectory. Middle: control magnitudes. Right: control plot.}
    \label{fig:foh}
\end{figure*}

\begin{figure*}[t]
    \centering
    \includegraphics[width=0.3\textwidth]{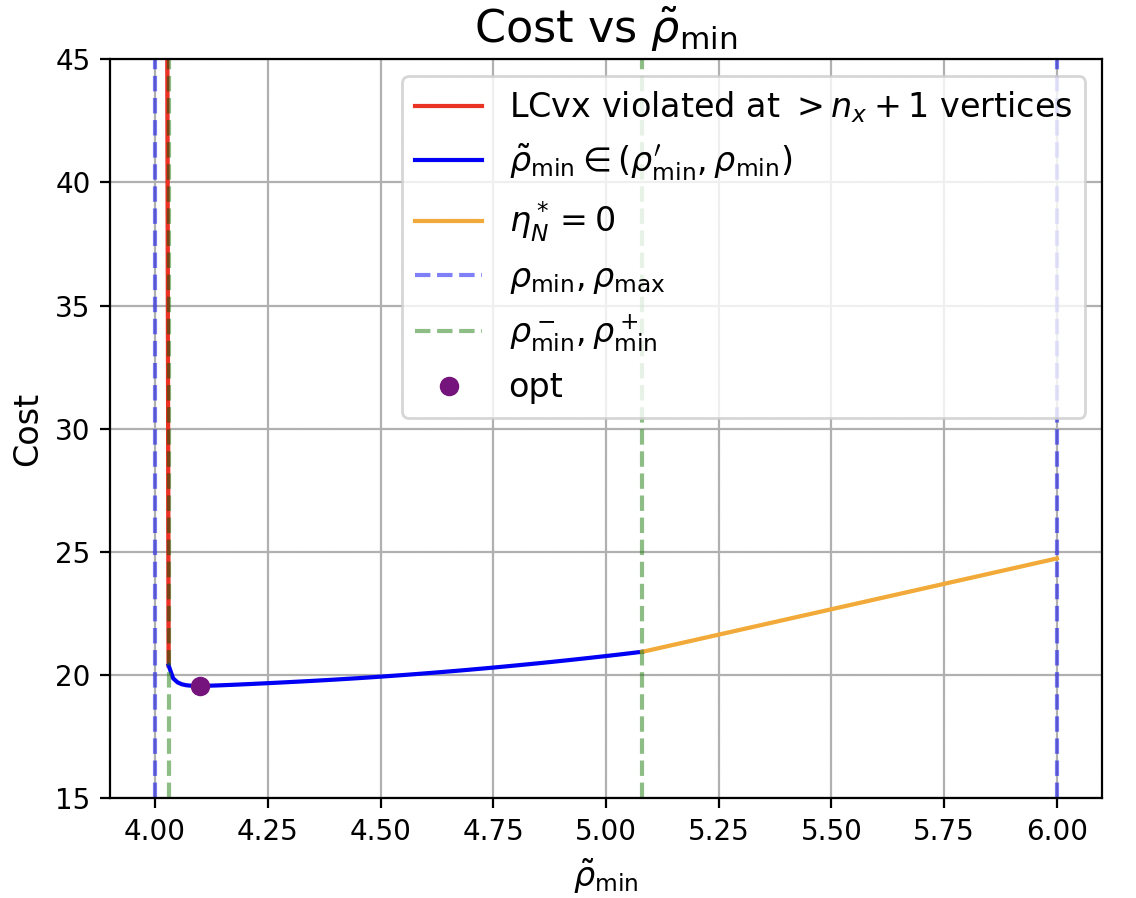}
    \includegraphics[width=0.3\textwidth]{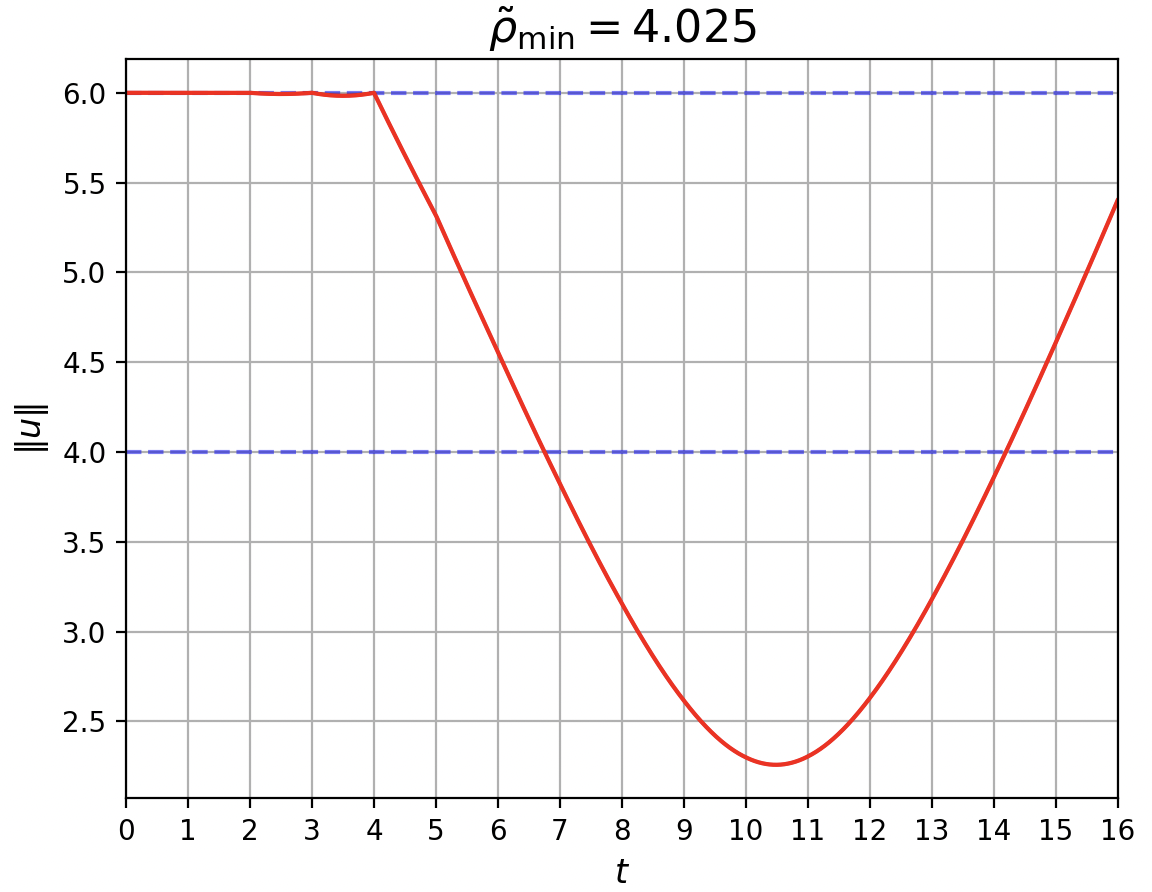}
    \includegraphics[width=0.3\textwidth]{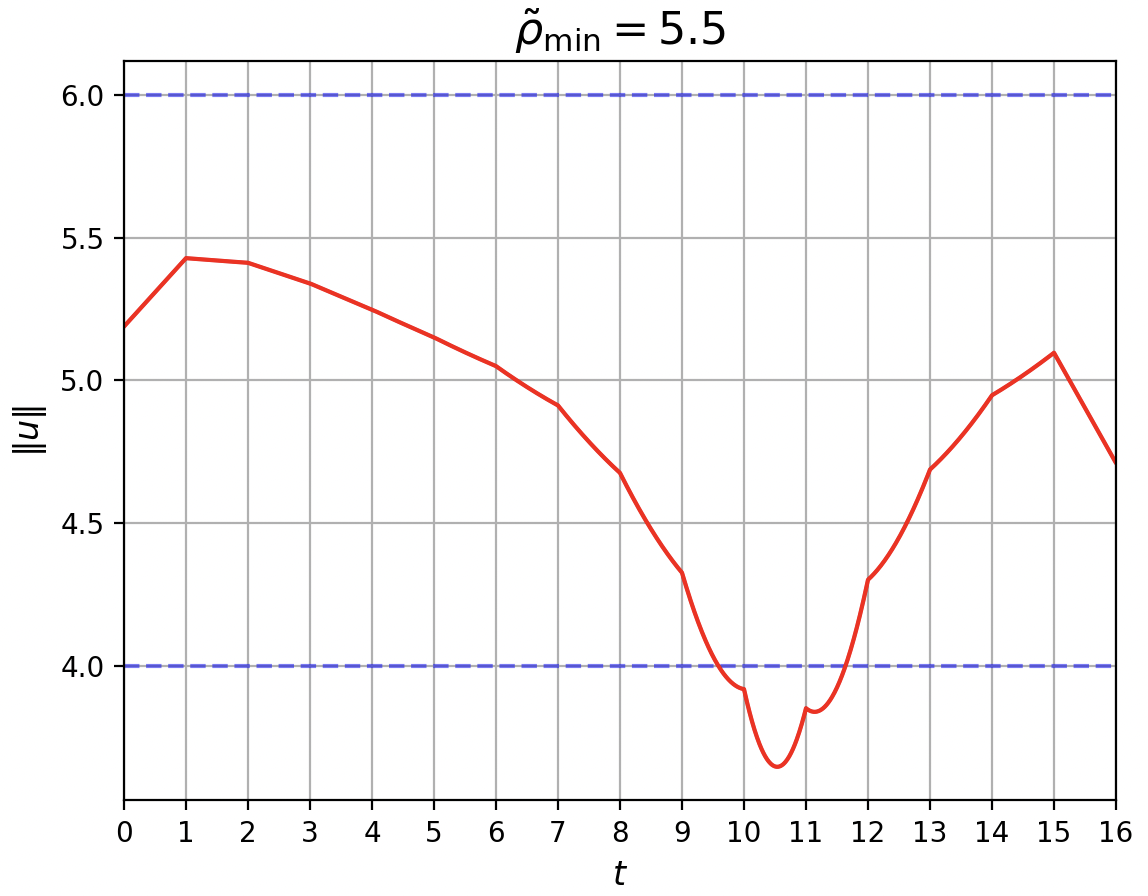}
    \caption{Left: cost vs $\tilde{\rho}_{\min}$. Middle: control magnitudes for $\tilde{\rho}_{\min} = 4.025$. Right: control magnitudes for $\tilde{\rho}_{\min} = 5.5$.}
    \label{fig:rho}
\end{figure*}

\section{Conclusion}

In this paper, we continue the line of work extending LCvx theory to the discrete-time setting, and in particular to optimal control problems with piecewise linear controls, addressing a significant gap in the existing LCvx literature and allowing a new class of nonconvex optimal control problems to be solved via numerical methods. Specifically, we present an algorithm that, under mild assumptions, finds an optimal trajectory where LCvx is violated along at most $2n_x + 2$ edges in $O(\log(\Delta\rho/\varepsilon))$ calls to our solver. The main insights of this paper were that given LCvx guarantees on the vertices, we can extend this guarantee to the edges by perturbing the control norm lower bound and adding a rate constraint on our controls. Finally, we showcase this algorithm on a classical numerical example, demonstrating the effectiveness of the methods outlined in this paper.

Future research directions include exploring other discrete-time nonconvex optimal control problems such as problems with different cost functionals, path constraints, and control parameterizations. More broadly, this work represents a step toward understanding the limitations of current LCvx theory and developing increasingly general algorithms capable of solving nonconvex optimal control problems with continuous-time constraint satisfaction in practice.

\section{Acknowledgments}
This work was completed as part of the University of Washington Autonomous Controls Laboratory under the supervision of Dr. Beh\c{c}et A\c{c}{\i}kme\c{s}e. Special thanks to Dayou Luo for his research that informed this work and his valuable feedback on this write-up.

\bibliographystyle{IEEEtran}
\bibliography{refs}

\begin{thebibliography}{10}
\providecommand{\url}[1]{#1}
\csname url@samestyle\endcsname
\providecommand{\newblock}{\relax}
\providecommand{\bibinfo}[2]{#2}
\providecommand{\BIBentrySTDinterwordspacing}{\spaceskip=0pt\relax}
\providecommand{\BIBentryALTinterwordstretchfactor}{4}
\providecommand{\BIBentryALTinterwordspacing}{\spaceskip=\fontdimen2\font plus
\BIBentryALTinterwordstretchfactor\fontdimen3\font minus \fontdimen4\font\relax}
\providecommand{\BIBforeignlanguage}[2]{{%
\expandafter\ifx\csname l@#1\endcsname\relax
\typeout{** WARNING: IEEEtran.bst: No hyphenation pattern has been}%
\typeout{** loaded for the language `#1'. Using the pattern for}%
\typeout{** the default language instead.}%
\else
\language=\csname l@#1\endcsname
\fi
#2}}
\providecommand{\BIBdecl}{\relax}
\BIBdecl

\bibitem{xiao}
W.~Xiao, N.~Mehdipour, A.~Collin, A.~Bin-Nun, E.~Frarzzoli, R.~Duintjer~Tebbens, and C.~Belta, ``Rule-based optimal control for autonomous driving,'' in \emph{International Conference on Cyber-Physical Systems}, 2021.

\bibitem{weber}
T.~Weber, \emph{Optimal Control Theory with Applications in Economics}.\hskip 1em plus 0.5em minus 0.4em\relax MIT Press, 2011.

\bibitem{barbu}
V.~Barbu and M.~Iannelli, ``Optimal control of population dynamics,'' \emph{Journal of Optimization Theory and Applications}, vol. 102, 1999.

\bibitem{nemmirovski}
A.~Nemirovski, ``Interior point polynomial time methods in convex programming,'' \emph{Lecture notes}, vol.~42, no.~16, pp. 3215--3224, 2004.

\bibitem{nocedal}
J.~Nocedal and S.~J. Wright, \emph{Numerical Optimization}, 2nd~ed.\hskip 1em plus 0.5em minus 0.4em\relax Springer, 2006.

\bibitem{acikmese}
B.~A\c{c}{\i}kme\c{s}e and S.~Ploen, ``A powered descent guidance algorithm for mars pinpoint landing,'' in \emph{AIAA Guidance, Navigation, and Control Conference and Exhibit}, 2005.

\bibitem{accikmecse2011lossless}
B.~A{\c{c}}{\i}kme{\c{s}}e and L.~Blackmore, ``Lossless convexification of a class of optimal control problems with non-convex control constraints,'' \emph{Automatica}, vol.~47, no.~2, pp. 341--347, 2011.

\bibitem{accikmecse2013lossless}
B.~A{\c{c}}{\i}kme{\c{s}}e, J.~M. Carson, and L.~Blackmore, ``Lossless convexification of nonconvex control bound and pointing constraints of the soft landing optimal control problem,'' \emph{IEEE transactions on control systems technology}, vol.~21, no.~6, pp. 2104--2113, 2013.

\bibitem{harris2014lossless}
M.~W. Harris and B.~A{\c{c}}{\i}kme{\c{s}}e, ``Lossless convexification of non-convex optimal control problems for state constrained linear systems,'' \emph{Automatica}, vol.~50, no.~9, pp. 2304--2311, 2014.

\bibitem{berkovitz}
L.~Berkovitz, \emph{Optimal Control Theory}.\hskip 1em plus 0.5em minus 0.4em\relax Springer, 1974.

\bibitem{cvxpy1}
S.~Diamond and S.~Boyd, ``{CVXPY}: {A} {P}ython-embedded modeling language for convex optimization,'' \emph{Journal of Machine Learning Research}, vol.~17, no.~83, pp. 1--5, 2016.

\bibitem{cvxpy2}
A.~Agrawal, R.~Verschueren, S.~Diamond, and S.~Boyd, ``A rewriting system for convex optimization problems,'' \emph{Journal of Control and Decision}, vol.~5, no.~1, pp. 42--60, 2018.

\bibitem{ecos}
A.~Domahidi, E.~Chu, and B.~Stephen, ``{ECOS}: An {SOCP} solver for embedded systems,'' in \emph{2013 European Control Conference (ECC)}, 2013, pp. 3071--3076.

\bibitem{luo}
D.~Luo, K.~Echigo, and B.~A{\c{c}}{\i}kme{\c{s}}e, ``Revisiting lossless convexification: Theoretical guarantees for discrete-time optimal control problems,'' \emph{Automatica}, vol. 183, 2026.

\bibitem{boyd}
S.~Boyd and L.~Vandenberghe, \emph{Convex Optimization}.\hskip 1em plus 0.5em minus 0.4em\relax Cambridge University Press, 2004.

\bibitem{borwein}
J.~Borwein and A.~Lewis, \emph{Convex Analysis}.\hskip 1em plus 0.5em minus 0.4em\relax Springer, 2006.

\bibitem{bauschke}
H.~Bauschke, \emph{Convex Analysis and Monotone Operator Theory in Hilbert Spaces}.\hskip 1em plus 0.5em minus 0.4em\relax Springer Science \& Business Media, 2011.

\bibitem{strang}
G.~Strang, \emph{Introduction to Linear Algebra, Fifth Edition}.\hskip 1em plus 0.5em minus 0.4em\relax Wellesley-Cambridge Press, 2016.

\end{thebibliography}

\appendix

\subsection{Proof of Proposition \ref{prop:conditions}.}
\begin{proof}
    From Lemma \ref{lem:duality}, we have
    $$u^* = \arg\min_{u} \mathcal{L}(x^*, u, \sigma^*, \eta^*, \mu_1^*, \mu_2^*)$$
    Isolating each $u_i$ and applying Lemma \ref{lem:etas}.1, we have\\
    $i = 1$:
    \begin{align*}
        u_1^* &= \arg\min_{u_1} I_V(u_1, \sigma_1^*) + \eta_1^{*\top} B_0 u_1\\
        &= \arg\min_{u_1} \eta_1^{*\top} B_0 u_1 \quad \textnormal{s.t. } u_1 \in V(\sigma_1^*)
    \end{align*}
    $2 \le i \le N$:
    \begin{align*}
        u_i^* &= \arg\min_{u_i} I_{V}(u_i, \sigma_i^*) + \eta_i^{*\top} B_0 u_i + \eta_{i-1}^{*\top} B_1 u_i\\
        &= \arg\min_{u_i} I_{V}(u_i, \sigma_i^*) + \eta_i^{*\top} B_0 u_i + (\eta_i^{*\top} A) B_1 u_i\\
        &= \arg\min_{u_i} \eta_i^{*\top} (B_0 + A B_1) u_i \quad \textnormal{s.t. } u_i \in V(\sigma_i^*)
    \end{align*}
    $i = N+1$:
    \begin{align*}
        u_{N+1}^* &= \arg\min_{u_{N+1}} I_V(u_{N+1}, \sigma_{N+1}^*) + \eta_N^{*\top} B_1 u_{N+1}\\
        &= \arg\min_{u_{N+1}} \eta_N^{*\top} B_1 u_{N+1} \quad \textnormal{s.t. } u_{N+1} \in V(\sigma_{N+1}^*)
    \end{align*}
    
    Since all of these are convex optimization problems, it follows from Proposition 2.1.2 of \cite{borwein} that\\
    $i = 1$:
    $$-\frac{\partial}{\partial u_1} \left[ \eta_1^{*\top} B_0 u_1 \right] = -B_0^\top \eta_1^* \in N_{V(\sigma_1^*)}(u_1^*)$$
    $2 \le i \le N$:
    \begin{align*}
        -\frac{\partial}{\partial u_i} \left[ \eta_i^{*\top} (B_0 + A B_1) u_i \right] &= -(B_0 + AB_1)^\top \eta_i^*\\
        &\in N_{V(\sigma_i^*)}(u_i^*)
    \end{align*}
    $i = N + 1$:
    $$-\frac{\partial}{\partial u_{N+1}} \left[ \eta_N^{*\top} B_1 u_{N+1} \right] = -B_1^\top \eta_N^* \in N_{V(\sigma_{N+1}^*)}(u_{N+1}^*)$$
    where $N_{V(\sigma_i^*)}(u_i^*)$ denotes the normal cone of convex set $V(\sigma_i^*)$ at point $u_i^*$.
    
    By Corollary 6.44 of \cite{bauschke}, $N_{V(\sigma_i^*)}(u_i^*)$ has nonzero elements only on the boundary of $V(\sigma_i^*)$. By Lemma \ref{lem:boundary}, we have our result.
\end{proof}

\subsection{Proof of Lemma \ref{lem:edges}.}
\begin{proof}
    Without loss of generality, assume $\|u_i\| \le \|u_{i+1}\|$. Recall that $u(t)$ linearly interpolates $u_i$ and $u_{i+1}$:
    $$u(t) = \left( \frac{t_{i+1} - t}{t_{i+1} - t_i} \right) u_i + \left(\frac{t - t_{i}}{t_{i+1} - t_i} \right) u_{i + 1}$$
    
    We now examine the minimum value of $\|u(t)\|$ for $t \in [t_i, t_{i+1}]$ and divide into two cases:
    
    Case 1: $\|u(t)\|$ is minimized at the endpoints. Since $\|u_i\| \le \|u_{i+1}\|$, this must be at $t = t_i$ and thus
    $$\|u(t)\| \ge \|u_i\| \ge \tilde{\rho}_{\min} \ge \rho_{\min}$$
    
    Case 2: $\|u(t)\|$ is not minimized at the endpoints. Then, by section 4.2 of \cite{strang}, the minimal value of $\|u(t)\|$ is the perpendicular distance from the origin to the line $L$ passing through $u_i$ and $u_{i+1}$. Let $\textnormal{proj}_L u_i$ denote the projection of $u_i$ onto $L$. Then
    \begin{align*}
        \|u(t)\| &\ge \|u_i - \textnormal{proj}_L u_i\|\\
        &= \sqrt{\|u_i\|^2 - \|\textnormal{proj}_L u_i\|^2}\\
        &= \sqrt{\|u_i\|^2 - \left\| \frac{u_i^\top (u_{i+1} - u_i)}{\|u_{i+1} - u_i\|^2} (u_{i+1} - u_i) \right\|^2}\\
        &= \sqrt{\|u_i\|^2 - \left( \frac{| u_i^\top (u_{i+1} - u_i) |}{\| u_{i+1} - u_i \|} \right)^2}
    \end{align*}
    Note that
    $$\|u_{i+1} - u_i\|^2 = \|u_{i+1}\|^2 + \|u_i\|^2 - 2 u_i^\top u_{i+1}$$
    By the triangle inequality,
    $$\|u_{i+1}\| \le \|u_i\| + \|u_{i+1} - u_i\|$$ 
    Combining the above gives
    \begin{align*}
        u_i^\top (u_{i+1} - u_i) &= u_i^\top u_{i+1} - \|u_i\|^2\\
        &= \frac{\|u_{i+1}\|^2 + \|u_i\|^2 - \|u_{i+1} - u_i\|^2}{2} - \|u_i\|^2\\
        &= \frac{\|u_{i+1}\|^2 - (\|u_i\|^2 + \|u_{i+1} - u_i\|^2)}{2}\\
        &\le 0
    \end{align*}
    Thus,
    $$|u_i^\top (u_{i+1} - u_i)| = \frac{\|u_{i+1} - u_i\|^2 + \|u_i\|^2 - \|u_{i+1}\|^2}{2}$$
    Since $\|u_i\| \le \|u_{i+1}\|$ and $0 < \rho_{\min} \le \tilde{\rho}_{\min}$,
    \begin{align*}
        \frac{| u_i^\top (u_{i+1} - u_i) |}{\| u_{i+1} - u_i \|} &= \frac{\|u_{i+1} - u_i\|^2 + \|u_i\|^2 - \|u_{i+1}\|^2}{2\|u_{i+1} - u_i\|}\\
        &\le \frac{\|u_{i+1} - u_i\|^2}{2\|u_{i+1} - u_i\|}\\
        &= \frac{\|u_{i+1} - u_i\|}{2}\\
        &\le \frac{\delta(\tilde{\rho}_{\min})}{2}\\
        &= \sqrt{\tilde{\rho}_{\min}^2 - \rho_{\min}^2}\\
        &\le \tilde{\rho}_{\min}\\
        &\le \|u_i\|
    \end{align*}
    Plugging this back into our initial inequality, we have
    \begin{align*}
        \|u(t)\| &\ge \sqrt{\tilde{\rho}_{\min}^2 - \left( \sqrt{\tilde{\rho}_{\min}^2 - \rho_{\min}^2} \right)^2}\\
        &= \rho_{\min}
    \end{align*}
    
    Finally, since $\| u \| \le \rho_{\max}$ is convex, if $\|u_i\| \le \rho_{\max}$ and $\|u_{i+1}\| \le \rho_{\max}$, then the linear interpolations $u(t)$ between them must obey $\|u(t)\| \le \rho_{\max}$ by definition. Thus, $\rho_{\min} \le \| u(t) \| \le \rho_{\max}$ for all $t \in [t_i, t_{i+1}]$.
\end{proof}

\end{document}